\documentclass[a4paper,11pt,reqno]{amsart}

\usepackage[utf8]{inputenc}
\usepackage[T1]{fontenc}

\usepackage{paralist}				
\usepackage{letterspace}
\usepackage{ifdraft}
\usepackage{cite}
\usepackage{microtype}
\usepackage{tikz-cd}

\usepackage{amsmath}
\usepackage{amssymb}
\usepackage{amsthm}				
\usepackage{mathtools}				
\usepackage{mathrsfs}				
\usepackage{dsfont}				

\usepackage{hyperref}

\newcommand{\scr}{\mathscr}

\newcommand{\IN}{\mathbb{N}}
\newcommand{\IZ}{\mathbb{Z}}
\newcommand{\IR}{\mathbb{R}}
\newcommand{\IC}{\mathbb{C}}

\newcommand{\cC}{\mathcal{C}}

\newcommand{\cE}{\mathcal{E}}
\newcommand{\cM}{\mathcal{M}}
\newcommand{\cP}{\mathcal{P}}

\newcommand{\cU}{\mathcal{U}}

\newcommand{\ga}{\alpha}
\newcommand{\gd}{\delta}

\newcommand{\gl}{\lambda}
\newcommand{\gi}{\iota}
\newcommand{\gk}{\kappa}
\newcommand{\gp}{\varphi}
\newcommand{\gs}{\sigma}
\newcommand{\gS}{\Sigma}

\newcommand{\ev}{\mathbf{E}}

\newcommand{\one}{\mathds{1}}
\newcommand{\supp}{\operatorname{supp}}
\newcommand{\diag}{\operatorname{diag}}

\newcommand{\Zsp}{\operatorname{span}_{\IZ}}

\newcommand{\Mat}{\operatorname{Mat}}
\newcommand{\Matp}{\operatorname{Mat}_p}
\newcommand{\IntMat}{\operatorname{IntMat}}
\newcommand{\IntMatp}{\operatorname{IntMat_p}}
\newcommand{\PU}{\operatorname{PU}}

\newcommand{\fmSet}{\operatorname{fmSet}}

\newcommand{\injto}{\hookrightarrow}
\newcommand{\ssq}{\subseteq}

\newcommand{\krs}{\boxplus}

\newcommand{\Polya}{P\'olya\ }

\renewcommand{\Re}{\operatorname{Re}}

\newcommand{\Di}{\,\operatorname{d}\!}
\newcommand{\what}{\widehat}

\newtheorem{lemma}{Lemma}
\newtheorem{cor}[lemma]{Corollary}
\newtheorem{thm}[lemma]{Theorem}
\newtheorem{prop}[lemma]{Proposition}
\theoremstyle{definition}
\newtheorem{defin}[lemma]{Definition}
\theoremstyle{remark}
\newtheorem{rem}[lemma]{Remark}

\title{A Semiring Structure for Generalised P\'olya Urns}
\author{Fabian Burghart}
\address{Department of Mathematics, Uppsala University, S-751 06 Uppsala, Sweden}
\email{fabian.burghart@math.uu.se}

\date{12 November 2021}
\keywords{Generalised P\'olya urns; intensity matrix; semiring morphism; Kronecker sum}
\subjclass[2020]{60C05 (Primary); 16Y60 (Secondary)} 

\begin{document}

\begin{abstract}
We define the notions of disjoint unions and products for generalised \Polya urns, proving that this turns the set of isomorphism classes of urns into a commutative semiring. The set of square matrices up to similarity by a permutation matrix is also a commutative semiring under the operations of direct sum and Kronecker sum, and we prove that assigning to an urn its intensity matrix leads to a morphism of semirings. 

Moreover, we show that a second semiring morphism exists, sending intensity matrices to their spectra. This, together with the existence of the first morphism has implications for the asymptotic behaviour of product urns, which are discussed. 
\end{abstract}

\maketitle

\section{Introduction}

\Polya urn models have, since the paper \cite{EP23} almost a century ago, undergone significant development and generalisation. Where Eggenberger and \Polya regarded urns in which balls have two colours, and upon drawing a ball, a fixed number of balls with the same colour are added to the urn, Bernstein \cite{Ber40} and Friedman \cite{Fri49} among others considered replacements that affect both colours. The type of model used in this paper was perhaps first suggested by Athreya and Karlin \cite[Remark 3.3]{AK68}: We allow for an arbitrary finite number of colours, and the replacements can themselves be random (according to some fixed probability distribution, and independently of everything else). 

However, even more general urns have been introduced and investigated, allowing for infinitely many colours (see the work of Bandyopadhyay and Thacker \cite{BT17} or Mailler and Marckert \cite{MM17}) or allowing for drawing multiple balls at once, for example in \cite{TM01}.

Typically, most work on \Polya urns focuses on obtaining limit theorems for the urn process (see, among others, \cite{AK68}, \cite{BP85}, \cite{smythe1996central}, \cite{Ja04}, \cite{flajolet2005analytic}, \cite{Ja06}, or \cite{Po08}) or on applying these theorems to obtain results in bioscience or computer science, see e.g. \cite{Ma02}, \cite{HJ15} for the latter. We also refer to Mahmoud's book \cite{Ma08} for a general overview of \Polya urns and their applicability. This paper takes a different viewpoint by focusing on the urn models themselves. Specifically, given two urn models $\cU$ and $\cU'$, we present two operations to construct new urns from them: The disjoint union $\cU\sqcup\cU'$ and the product $\cU\times\cU'$, where the names reflect the set-theoretic operations on the underlying colour sets. 

Since stochastic urn models tend to have a rather combinatorial nature, it is perhaps no surprise that they should form the structure of a semiring, thereby reflecting the algebraic structure of the non-negative integers, and of combinatorial classes under disjoint unions and Cartesian products \cite[4.2.3.4]{Ba09}. Compare also Example 1.4.9 in \cite{Ri17}, where it is shown how the category of finite sets with isomorphisms (i.e. bijections) is a categorification of the non-negative integers.

\subsection*{Main result} 

The main purpose of this paper is the following theorem, providing a pipeline that translates constructions on the level of \Polya urns to statements about their intensity matrices and the sets of eigenvalues thereof: 
\begin{thm}\label{thm:pipeline}
 The following sets, together with the listed operations as addition and multiplication respectively, form commutative semirings: 
 \begin{itemize}
  \item The set $\PU$ of strict isomorphism classes of \Polya urns, together with disjoint unions and products
  \item The set $\IntMatp$ of equivalence classes of intensity matrices under permutation similarity, together with direct sums and Kronecker sums
  \item The set $\fmSet_\IC$ of finite multisets of complex numbers, together with disjoint unions and Minkowski sums. 
 \end{itemize}
 Moreover, there are morphisms $\Phi$ and $\sigma$ of semirings as follows 
 \begin{center}
  \begin{tikzcd}
   \PU \arrow[r, "\Phi"] & \IntMatp \arrow[r, "\sigma"] & \fmSet_\IC
  \end{tikzcd}
 \end{center}
 where $\Phi$ assigns to an isomorphism class of \Polya urns its equivalence class of intensity matrices, and $\sigma$ maps an equivalence class of activity matrices to the spectrum of a representative. 
\end{thm}

After giving a definition of (generalised) \Polya urns in Section~\ref{sec:prelim}, we will introduce the relevant structure for $\PU$ in Section~\ref{sec:semiring} and show that this structure indeed forms a commutative semiring in Theorem~\ref{thm:PUsr}. The set $\IntMatp$ will be properly defined in Definition~\ref{defin:IntMatp}, and proven to be a semiring in the appendix. The remaining statements will be verified in Section~\ref{sec:intensity}. 

Finally, Section~\ref{sec:properties} discusses some of the consequences of Theorem~\ref{thm:pipeline} with regards to properties that are of interest when proving limit theorems, and Section~\ref{sec:examples} illustrates two special cases.

\section{Preliminaries} \label{sec:prelim}

\subsection*{Some (linear) algebra} Let $R\in\{\IZ,\IR\}$. Recall that for any set $Q$, there exists the free $R$-module generated by $Q$, consisting of functions $Q\to R$ which are $0$ at almost all points (for $R=\IR$ this is merely the real vector space with basis $Q$). We denote this module by $R^Q$. We will not distinguish between a vector in $R^Q$ and a map $Q\to R$ if $Q$ is a finite set (although the usual notation of a vector as a column or row vector would additionally require the choice of a linear order on the set $Q$ -- we circumvent this by assuming that $Q$ comes equipped with one). For a vector $v\in R^Q$, we denote the entries by $v_i, i\in Q$. Moreover, any map $\gp:Q\to Q'$ functorially induces an  $R$-linear map $\hat\gp:R^Q\to R^{Q'}$ via
\begin{center}
 \begin{tikzcd}
  Q \arrow[d, "\gp"'] \arrow[rd, "\iota'\circ\gp" description] \arrow[r, "\iota", hook] & R^Q \arrow[d, "\exists!\hat\gp", dashed] \\
  Q' \arrow[r, "\iota'"', hook]                                                         & R^{Q'}                                  
 \end{tikzcd}
\end{center}
using the universal property of free modules, where $\iota$ denotes the embedding of $Q$ as a basis into $R^Q$ (by sending $i\in Q$ to the map $\mathds{1}_{\{i\}}$), and where $\iota'$ is the analogous map for $Q'$. 

Given a free $R$-module $R^Q$, its non-negative cone $\{x\in R^Q:x_i\geq0\ \forall i\in Q\}$ will be denoted by $R_{\geq0}^Q$.

We will make frequent use of the direct sum and the Kronecker product of real matrices. Given $A\in \IR^{n\times m}$ and $B\in \IR^{n'\times m'}$, we have the respective block forms:
\[
 A\oplus B:=\begin{pmatrix}
             A & 0 \\
             0 & B
            \end{pmatrix}  \in \IR^{(n+n')\times(m+m')}
\]
and 
\[
 A\otimes B:=\begin{pmatrix}
             A_{11}B & \hdots & A_{1m}B\\
             \vdots  & \ddots & \vdots \\
             A_{n1}B & \hdots & A_{nm}B
            \end{pmatrix} \in \IR^{nn'\times mm'}.
\]
For square matrices $A\in\IR^{n\times n}, B\in\IR^{m\times m}$ we can also define the Kronecker sum via 
\[
 A\krs B:=A\otimes I_m + I_n\otimes B,
\]
where $I_k$ denotes the unit matrix of size $k, k\geq0$. For elementary properties of $\otimes$ and $\krs$, cf. e.g. \cite{BC08} and \cite{HJ91}. For two vectors $a,b$, by a slight abuse of notation, we also denote by $a\krs b$ the vector whose entries are all pairwise sums of entries from $a$ and $b$. More precisely, if $\diag(v)$ is the diagonal matrix having the vector $v$ on its main diagonal, then $a\krs b$ refers to the diagonal vector of the matrix $\diag(a)\krs\diag(b)$ which is itself diagonal. Observe that by definition, $a\krs b$ is the vector containing all possible entries of the form $a_i+b_j$, arranged in lexicographic order. 

For two finite sets $Q,Q'$ and two maps $a:Q\to R$ and $b:Q'\to R$, there is a canonically induced map $Q\sqcup Q'\to R$, where $Q\sqcup Q'$ denotes the disjoint union of $Q$ and $Q'$. Moreover, as $R$-modules we have $R^{Q\sqcup Q'} \cong R^Q\oplus R^{Q'}$, and will therefore denote by $a\oplus b$ the map induced by $a$ and $b$. If the linear orders on $Q$ and $Q'$ are extended to $Q\sqcup Q'$ in such a way that all elements of $Q$ come before any element of $Q'$, then $a\oplus b$ is the vector obtained by concatenating the vectors $a$ and $b$. 

Finally, we remark that the usual embedding of $\IZ$ in $\IR$ enables us to view any vector in $\IZ^Q$ also as a vector in $\IR^Q$, and we will occasionally use this.


\subsection*{\Polya urns}

By a (generalised) \Polya urn, we mean a tuple 
\[
  \cU=\left(Q,\mu=(\mu_i)_{i\in Q},a,X_0\right),
\]
consisting of the following data:
\begin{itemize}
  \item A finite set $Q$ whose elements are called types or colours,
  \item A collection $\mu$ of probability measures $\mu_i\in \cM_P(\IZ^Q)$ for $i\in Q$, supported within
  \begin{equation}\label{eq:supp}
    S_{Q,i}:=\left\{x\in \IZ^Q:x_j\geq -\gd_{ij}\quad \forall j\in Q\right\},
  \end{equation}
  \item A map $a\in \IR_{\geq0}^Q$ assigning to each colour its activity,
  \item A map $X(0)\in \IZ_{\geq0}^Q$ called the initial configuration,
\end{itemize}
such that for any $i\in Q$ with $a_i=0$, we have $\mu_i=\gd_0$, the Dirac measure concentrated at $0\in\IZ^Q$. 

Here, we used the notation $\cM_P(\IZ^Q)$ for the set of probability measures on $\IZ^Q$ with the $\sigma$-algebra being implied to be the power set of $\IZ^Q$.

For $n\in\IN_0,i\in Q$, define independent $\mu_i$-distributed random vectors $\xi_i(n)=\left(\xi_{i,j}(n)\right)_{j\in Q}$.
We can define the corresponding \Polya urn process to be the $\IZ_{\geq0}^Q$-valued Markov process $(X(n))_{n\in\IN_0}$ starting in the initial configuration, where at every time $n=1,2,\dots$, a ball of colour $i$ is drawn with probability $\frac{a_iX_i(n)}{\left<a,X(n)\right>}$, after which the drawn ball and $\xi_{i,j}(n)$ balls of colour $j$ for each $j\in Q$ are added to the urn. Here, $\left<\cdot\ ,\ \cdot\right>$ denotes the standard scalar product on $\IR^Q$. The process is stopped as soon as $\left<a,X(n)\right>=0$ in which case we say that the urn has become \emph{essentially extinct}. This process is defined uniquely in distribution by the data of a \Polya urn.

The \Polya urn process can be thought of as having an urn containing $X_i(n)$ balls of type $i$ at time $n$. From this urn, one ball is drawn at random with probability proportional to the activity of its type. If the drawn ball is of type $i$, then, independently of everything else, the ball itself together with $\xi_{i,j}(n)$ balls of type $j$ are added back into the urn for all types $j\in Q$, where we interpret $\xi_{i,i}(n)=-1$ as removing the drawn ball again. Note that the condition $\xi_{i,j}(n)\geq -\gd_{ij}$ implies that we never remove balls other than the one drawn from the urn, therefore ensuring that the amount of balls of each type always stays non-negative. Finally, the stopping condition $\left<a,X(n)\right>=0$ occurs when there are no balls of positive activity in the urn.


\section{The semiring of \Polya urns} \label{sec:semiring}

\subsection*{Embeddings of urns}

For the rest of the paper, unless the urns in question are specified concretely, we will consider two \Polya urns $\cU$ and $\cU'$, with the following data:
\begin{equation}\label{eq:2urns}
 \cU=\left(Q,\mu=(\mu_i)_{i\in Q},a,X(0)\right)\ \text{ and }\ 
 \cU'=\left(Q',\mu'=(\mu'_i)_{i\in Q'},a',X'(0)\right)
\end{equation}
\begin{defin}\label{defin:emb}
 We say that $\cU$ embeds strictly into $\cU'$ if there exists an injection $\gp:Q\injto Q'$ subject to the following three conditions:
 \begin{enumerate}[(i)]
  \item The transition measures satisfy $\mu'_{\gp(i)}=\hat\gp_* \mu_i$ for all $i\in Q$. 
  \item The activities are compatible via $a'_{\gp(i)}=a_i$ for all $i\in Q$. 
  \item The initial configurations satisfy $X'_{\gp(i)}(0)=X_i(0)$ for all $i\in Q$. 
 \end{enumerate}
 Here, $\hat\gp_* \mu_i$ is the push-forward measure of $\mu_i$ along $\hat\gp$.
\end{defin}

Observe that the class of all \Polya urns form a category, where the morphisms are given by strict embeddings. It is easily verified that two urns are isomorphic in this category iff $\gp$ is a bijection, in which case we call the urns in question \emph{strictly isomorphic}.
In other words, two \Polya urns are strictly isomorphic if their only distinction is the labelling of their colours. The class of all strict isomorphism classes of \Polya urns will be denoted by $\PU$.

\begin{rem}
 The attribute ``strictly'' has the following reason: It is possible to define embeddings of \Polya urns via their urn processes. Indeed, if $Y$ is any $\IZ^Q_{\geq0}$-valued Markov process starting in $X(0)$ and adapted to a filtration $\scr F$, then we could say that $X$ embeds into $Y$ if there is an increasing sequence of $\scr F$-stopping times $(\tau_n)_{n\in\IN}$ such that the processes $(X_n)_{n\in\IN}$ and $\left(Y_{\tau_n}\right)_{n\in\IN}$ coincide in distribution. One instance of this is the Athreya-Karlin embedding \cite{AK68}.
\end{rem}

\subsection*{Zero and unit urn}
The above notion of \Polya urn does not prevent us from considering rather trivial urn models: For example, choosing $Q=\emptyset$ determines a unique urn containing zero balls of no colours. We will call this the \emph{zero urn} and denote it by $\cP_\emptyset$. Similarly, for a singleton set $Q=\{q\}$, we can choose $\mu_q=\gd_0$, $a_q=0$ and $X(0)=1$. We will call this the \emph{unit urn} and denote it by $\cP_0$ (The zero in the index reflects the activity of the colour, see also section~\ref{ssec:slow} below). Note that both $\cP_\emptyset$ and $\cP_0$ are essentially extinct already at time zero. In this regard, the choice $X(0)=1$ may seem arbitrary, but it ensures that $\cP_0$ behaves as multiplicative unit for a product operation defined below.

\subsection*{Disjoint unions} 
The disjoint union of two \Polya urns should be thought of as pouring all the balls from $\cU$ and $\cU'$ into one larger, common urn from which balls are drawn. If the drawn colour is from $\cU$ proceed just as one would have done in $\cU$, and analogously, if the drawn colour is from $\cU'$ proceed as one would have done in $\cU'$.

Formally, note first that the canonical injections $\gi_Q:Q\injto Q\sqcup Q'$ and $\gi_{Q'}:Q'\injto Q\sqcup Q'$ into the disjoint union of $Q$ and $Q'$ induce the following injections on the set of probability measures 
\begin{align*}
 (\hat\gi_Q)_*&:\cM_P\left(\IZ^Q\right)\injto\cM_P\left(\IZ^{Q\sqcup Q'}\right)\ 
 \text{ and}\\ 
 (\hat\gi_{Q'})_*&:\cM_P\left(\IZ^{Q'}\right)\injto\cM_P\left(\IZ^{Q\sqcup Q'}\right)
\end{align*}
which allow us to transport the measures in $\mu\sqcup\mu'$ to $\IZ^{Q\sqcup Q'}=\IZ^Q\times \IZ^{Q'}$. Following this construction, we define 
\[
  \mu^\sqcup:=\big((\hat\gi_Q)_*\mu_i\big)_{i\in Q} \sqcup \big((\hat\gi_{Q'})_*\mu'_i\big)_{i\in Q'},
\]
which is a family of probability measures on $\IZ^{Q\sqcup Q'}$ indexed by $Q\sqcup Q'$. For $i\in Q\sqcup Q'$, we moreover have $\supp \mu^\sqcup_i\ssq S_{Q\sqcup Q',i}$ in the notation of \eqref{eq:supp}. This can be seen by noting that e.g. $\supp\mu^\sqcup_i=S_{Q,i}\times\{0\}\ssq S_{Q\sqcup Q',i}$ for $i\in Q$, and similarly for $i\in Q'$. Finally, we define the activities to be $a^\sqcup := a\oplus a'$, and the initial configuration to be $X^\sqcup(0):=X(0)\oplus X'(0)$

\begin{defin}\label{defin:sqcup}
 Using the notation just introduced, we call
 \[
  \cU\sqcup\cU' := \left(Q\sqcup Q', \mu^\sqcup, a^\sqcup, X^\sqcup(0)\right)
 \]
 the disjoint union of the \Polya urns $\cU$ and $\cU'$. 
\end{defin}

Observe that this is a well-defined operation on $\PU$: If $\gp:Q\to\tilde Q$ and $\gp':Q'\to\tilde Q'$ yield strict isomorphisms between \Polya urns $\cU\cong\tilde\cU$ and $\cU'\cong\tilde\cU'$, respectively, then the induced map $(\gp\sqcup\gp'):Q\sqcup Q'\to \tilde Q\sqcup \tilde Q'$ induces a strict isomorphism $\cU\sqcup\cU'\cong \tilde\cU\sqcup\tilde\cU'$.

\subsection*{Products}

In contrast to the preceding construction, the product of $\cU$ and $\cU'$ will contain balls coloured by pairs of colours $(q,q')\in Q\times Q'$, and if such a ball is drawn it either evolves according to its first coordinate (meaning that the replacements for $q$ from $\cU$ happen, but all balls involved will have $q'$ as the second coordinate in their colours) or according to its second coordinate (which works analogously), and the choice between the two options being random proportional to the activities of $q$ and $q'$. 

In order to construct such a \Polya urn from $\cU$ and $\cU'$ with colour set $Q^\times=Q\times Q'$, we define activities $a^\times:Q\times Q'\to \IR_{\geq0}$ given by $a_{(i,j)}=a_i+a'_j$ (so that $a^\times=a\krs a'$ up to a reordering of $Q^\times$), and initial configurations $X^\times_{(i,j)}(0)=X_i(0)X'_j(0)$ for $(i,j)\in Q\times Q'$, and it remains to define the probability measures in $\mu^\times$. 

Unlike for the canonical maps $\gi_Q$ and $\gi_{Q'}$ for the disjoint union $Q\sqcup Q'$, we are not interested in the projection maps $\pi_Q:Q\times Q'\to Q$ and $\pi_{Q'}:Q\times Q'\to Q'$ in order to realize product urns. Instead, assume that a ball of type $(i,j)$ has been drawn from $\cU\times\cU'$. Since one of the two coordinates of $(i,j)$ will remain fixed according to the description we gave at the beginning of this section, $\mu^\times_{(i,j)}$ should only be supported on $\IZ^{\{i\}\times Q'\cup Q\times\{j\}}$, which is a submodule of $\IZ^{Q\times Q'}$ of rank $|Q|+|Q'|-1$. For brevity, but also since this is the wedge sum of (discrete) topological spaces (cf. \cite[p.153]{rotman13}), we write $Q\vee_{i,j} Q'$ for $\{i\}\times Q'\cup Q\times\{j\}$. So, consider the maps 
\begin{align*}
 \gk_Q&:Q\ni k\mapsto (k,j)\in Q\vee_{i,j} Q'\ \text{ and }\\
 \gk_{Q'}&:Q'\ni k\mapsto (i,k)\in Q\vee_{i,j} Q'
\end{align*}
which are also the maps realising $Q\vee_{i,j} Q'$ as the following pushout in the category of sets:
\begin{center}
 \begin{tikzcd}
  \{*\} \arrow[r, "*\mapsto i", hook] \arrow[d, "*\mapsto j"', hook] \arrow[rd, phantom, "\lrcorner", very near end] & Q \arrow[d, "\gk_Q"]           \\
  Q' \arrow[r, "\gk_{Q'}"']                                                     & Q\vee_{i,j} Q'
 \end{tikzcd}
\end{center}
Recall that this means (by the definition of a pushout, see e.g. \cite[p.68f]{ML98}) that for any other set $S$ with maps $\ga_Q:Q\to S$ and $\ga_{Q'}:Q'\to S$ such that $\ga_Q(i)=\ga_{Q'}(j)$, there is a unique $\beta:Q\vee_{i,j} Q' \to S$ such that the following diagram commutes: 
\begin{center}
 \begin{tikzcd}
  \{*\} \arrow[r, "*\mapsto i", hook] 
        \arrow[d, "*\mapsto j"', hook]
  & Q \arrow[d, "\gk_Q"']   
      \arrow[bend left, rdd, "\ga_Q"] 
  & \\
  Q' \arrow[r, "\gk_{Q'}"]
     \arrow[bend right, rrd, "\ga_{Q'}"']
  & Q\vee_{i,j} Q' \arrow[rd, "\exists!\beta" description, dashed] 
  & \\
  & & S
 \end{tikzcd}
\end{center}
The maps $\gk_Q$ and $\gk_{Q'}$ induce a transport of probability measures
\begin{align*}
 (\what{\gk_Q})_*&:\cM_P\left(\IZ^Q\right)\injto\cM_P\left(\IZ^{Q\vee_{i,j} Q'}\right)\ \text{ and}\\ 
 (\what{\gk_{Q'}})_*&:\cM_P\left(\IZ^{Q'}\right)\injto\cM_P\left(\IZ^{Q\vee_{i,j} Q'}\right)
\end{align*}
and we can now define $\mu^\times_{(i,j)}$ via
\begin{equation}\label{eq:xmeas}
 \mu_{(i,j)}^\times = \frac{a_i}{a_i+a'_j}(\what{\gk_Q})_*\mu_i + \frac{a'_j}{a_i+a'_j}(\what{\gk_{Q'}})_*\mu'_j
\end{equation}
unless $a_i+a'_j=0$, in which case we define it to be the Dirac measure concentrated at 0, i.e. $\mu_{(i,j)}^\times := \gd_0$. Observe that this construction is consistent with the requirement \eqref{eq:supp} for the support of $\mu^\times_{(i,j)}$.

\begin{defin}\label{defin:x}
 Setting $\mu^\times=\left(\mu_{(i,j)}^\times\right)_{(i,j)\in Q\times Q'}$ and using the notation introduced above, we call
 \[
  \cU\times\cU' = \left(Q\times Q', \mu^\times, a^\times, X^\times(0)\right)
 \]
 the product urn of $\cU$ and $\cU'$.
\end{defin}

\begin{prop}\label{prop:xwell}
 The product of \Polya urns, as constructed in Definition \ref{defin:x}, is a well-defined operation on $\PU$.
\end{prop}
\begin{proof}
 Consider 4 urns $\cU,\cU',\tilde\cU$ and $\tilde\cU'$ with the accordingly labelled data, together with strict isomorphisms $\gp:\cU\to\tilde\cU$ and $\gp':\cU'\to\tilde\cU'$. We need to show $\cU\times\cU'\cong \tilde\cU\times\tilde\cU'$. 
 
 It is straightforward to verify that $\gp\times\gp':Q\times Q'\to \tilde Q\times\tilde Q'$ is a bijection, that $\tilde a^\times_{(\gp\times\gp')(i,j)}=a^\times_{(i,j)}$ and that $\tilde X^\times_{(\gp\times\gp')(i,j)}(0)=X^\times_{(i,j)}(0)$ for all $i\in Q,j\in Q'$. 
 
 For the compatibility of the measures, observe first that since the front and back squares in the following cube are pushouts, there exists a unique bijection $\gp^\times$ such that the entire cube commutes: 
 \begin{center}
 \begin{tikzcd}
  \{*\} \arrow[dd, "j"'] \arrow[rr, "i"] \arrow[equal,rd] &                                                   & Q \arrow[dd, "\kappa_Q"', near end] \arrow[rd, "\gp"]                          &                                                           \\
                                                   & \{*\} \arrow[dd, "\gp'(j)", near end] \arrow[rr, "\gp(i)"', near end]  &                                                                     & \tilde Q \arrow[dd, "\tilde\kappa_{\tilde Q}"]            \\
  Q' \arrow[rr, "\kappa_{Q'}", near start] \arrow[rd, "\gp'"']  &                                                   & Q\vee_{i,j} Q' \arrow[rd, "\exists!\gp^\times", dashed] &                                                           \\
                                                   & \tilde Q' \arrow[rr, "\tilde\kappa_{\tilde Q'}"'] &                                                                     & \tilde Q\vee_{\gp(i),\gp'(j)} \tilde Q'
 \end{tikzcd}
 \end{center}
 It is easy to check that the restriction
 \[
  \gp_{res}:=\gp\times\gp'|_{Q\vee_{i,j}Q'}^{\tilde Q\vee_{\gp(i),\gp'(j)}\tilde Q'}:Q\vee_{i,j} Q'\to \tilde Q\vee_{\gp(i),\gp'(j)} \tilde Q'
 \]
 can be used as a bijection in place of $\gp^\times$ so that therefore $\gp_{res}=\gp^\times$. Therefore, the induced push-forwards $\left(\what{\gp_{res}}\right)_*$ and $\left(\what{\gp^\times}\right)_*$ need to coincide as well. Using this, we obtain for $a_i+a'_j>0$ 
 \begin{align*}
  \tilde\mu^\times_{(\gp\times\gp')(i,j)}
  &= \frac{\tilde a_{\gp(i)}}{\tilde a_{\gp(i)}+\tilde a'_{\gp'(j)}}
      \left(\what{\tilde\gk_{\tilde Q}}\right)_*\tilde\mu_{\gp(i)}
    +\frac{\tilde a'_{\gp'(j)}}{\tilde a_{\gp(i)}+\tilde a'_{\gp'(j)}}
      \left(\what{\tilde\gk_{\tilde Q'}}\right)_*\tilde\mu'_{\gp'(j)}\\
  &= \frac{a_i}{a_i+a'_j}
      \left(\what{\tilde\gk_{\tilde Q}}\right)_*\hat\gp_*\mu_{i}
    +\frac{a'_j}{a_i+a'_j}
      \left(\what{\tilde\gk_{\tilde Q'}}\right)_*\hat\gp'_*\mu'_j.
 \end{align*}
 Here, due to the commutativity of the above diagram, we can apply
 \begin{align*}
  &\left(\what{\tilde\gk_{\tilde Q}}\right)_*\hat\gp_*
  =\left(\what{\gp^\times}\right)_*\left(\what{\gk_Q}\right)_*
  =(\what{\gp_{res}})_*\left(\what{\gk_Q}\right)_* \ \ \text{ and}\\
  &\left(\what{\tilde\gk_{\tilde Q'}}\right)_*\hat\gp'_*
  =(\what{\gp^\times})_*\left(\what{\gk_{Q'}}\right)_*
  =(\what{\gp_{res}})_*\left(\what{\gk_{Q'}}\right)_*
 \end{align*}
 to obtain 
 \begin{align*}
  \tilde\mu^\times_{(\gp\times\gp')(i,j)}
  &= \frac{a_i}{a_i+a'_j}(\what{\gp_{res}})_*\left(\what{\gk_Q}\right)_*\mu_i
    +\frac{a'_j}{a_i+a'_j}(\what{\gp_{res}})_*\left(\what{\gk_{Q'}}\right)_*\mu'_j\\
  &= (\what{\gp_{res}})_*\left(
       \frac{a_i}{a_i+a'_j}\left(\what{\gk_Q}\right)_*\mu_i
      +\frac{a'_j}{a_i+a'_j}\left(\what{\gk_{Q'}}\right)_*\mu'_j\right)\\
  &= (\what{\gp_{res}})_*\mu^\times_{(i,j)}
   = \left(\what{\gp\times\gp'}\right)_*\mu^\times_{(i,j)}
 \end{align*}
 as required (for the last step, we used that $\gp_{res}$ was merely a restriction of $\gp\times\gp'$ to the support of $\mu^\times_{(i,j)}$). The case where $a_i+a'_j=0$ can be treated analogously.
\end{proof}

\begin{thm}\label{thm:PUsr}
 The set $\PU$, equipped with $\sqcup$ as addition and $\times$ as multiplication is a commutative semiring, where $\cP_\emptyset$ and $\cP_0$ are the corresponding neutral elements.
\end{thm}
\begin{proof}
 Associativity, commutativity, and $\cP_\emptyset$ being neutral for $\sqcup$ can easily be seen from direct computation. The same holds true for $\cP_0$ being the neutral element for $\times$ and commutativity of multiplication. 
 
 For associativity of $\times$ and the distributive laws, consider three urns $\cU,\cU'$ and $\cU''$, where the data belonging to $\cU''$ is decorated by $''$, in analogy to \eqref{eq:2urns}. For brevity, denote by $\nu$ and $\nu'$ the collections of transition measures from $\cU\times\cU'$ and $\cU'\times\cU''$, respectively. Since $Q\times (Q'\times Q'')=(Q\times Q')\times Q''$ implies the desired equalities for the weights and initial distributions of $\cU\times(\cU'\times\cU'')$ and $(\cU\times\cU')\times\cU''$, only the equality of the measures is left to check. However, the replacement distribution for $(i,j,k)\in Q\times(Q'\times Q'')$ in the generic case where both $a_i+a'_j>0$ and $a'_j+a''_k>0$ is -- according to \eqref{eq:xmeas} -- given by
 \begin{multline*}
  \frac{a_i}{a_i+(a'_j+a''_k)}\left(\what{\gk_Q}\right)_*\mu_i
  +\frac{a'_j+a''_k}{a_i+(a'_j+a''_k)}\left(\what{\gk_{Q'\times Q''}}\right)_*\nu'_{(j,k)}\\
  =\frac{1}{a_i+a'_j+a''_k}\left(a_i\left(\what{\gk_Q}\right)_*\mu_i
   +a'_j\left(\what{\gk_{Q'}}\right)_*\mu'_j
   +a''_k\left(\what{\gk_{Q''}}\right)_*\mu''_k\right)
 \end{multline*}
 and the same expression arises for the replacement measure for $(i,j,k)$ in $(\cU\times\cU')\times\cU''$. The case where $a_i+a'_j=0$ or $a'_j+a''_k=0$ follows analogously. This shows associativity, and the distributive law is verified similarly. 
\end{proof}

\section{Intensity as a morphism of semirings} \label{sec:intensity}

For the analysis of \Polya urns the intensity matrix, generalising the transition matrix from finite-state Markov chains, is paramount. For an urn $\cU$ with $Q=\{1,...,q\}$ without loss of generality, this matrix is defined as
\begin{equation}
 A:=(a_j\ev\xi_{ji})_{i,j=1}^q
\end{equation}
which is a $q\times q$-matrix whose $(i,j)$-th entry is the product of $a_j$ and the expected change of the amount of type $i$ balls when a ball of type $j$ is drawn. (Here, we follow the convention in \cite{Ja04} as sometimes, the transpose of this matrix is taken as intensity matrix). Extending our choice of notation from \eqref{eq:2urns}, we denote by $A$ and $A'$ the intensity matrices of $\cU$ and $\cU'$, respectively. 

It immediately follows from the definition that all non-diagonal entries in $A$ are non-negative, and that $A_{ii}\geq -a_i$ for all $i\in Q$. Moreover, relabelling the colours of $\cU$ might lead to permuting the entries of $A$, and every possible such permutation arises by conjugating $A$ with a $q\times q$ permutation matrix. We hence define two $q\times q$-matrices $B, C$ to be \emph{permutation similar}, written $B\sim_p C$, whenever there exists a permutation matrix $P$ such that $C=P^{-1}BP$. By abuse of notation, we will not distinguish between a matrix and its equivalence class of permutation similar matrices, and denote both by the same symbol.

\begin{defin}\label{defin:IntMatp}
 Set $\Mat := \bigsqcup_{q\geq0} \IR^{q\times q}$, the set of all square matrices, and let $\IntMat := \{A\in\Mat: A_{ij}\geq0 \text{ whenever } i\neq j\}$. Finally, we define $\IntMatp := \IntMat / \sim_p$, which is well-defined since permutation similarity is an equivalence relation.
\end{defin}

The operations of direct sum and Kronecker sum make $\IntMatp$ into a semiring, see Corollary~\ref{cor:Intmp} in the appendix. Even more, we have the following result:

\begin{thm}\label{thm:Intense}
 The map $\Phi:\PU\to \IntMatp$ assigning to a \Polya urn the equivalence class of its intensity matrix is a morphism of semirings.
\end{thm}

The proof of Theorem~\ref{thm:Intense} relies on constructing suitable orderings on $Q\sqcup Q'$ and $Q\times Q'$ to obtain representatives $A^\sqcup$ for $\Phi(\cU\sqcup\cU')$ and $A^\times$ for $\Phi(\cU\times\cU')$. 

\begin{proof}
 It is straightforward to verify that $\Phi(\cP_\emptyset)=(\ )$ and $\Phi(\cP_0)=(0)$, so that $\Phi$ preserves neutral elements. For additivity, order the colours in $\cU\sqcup\cU'$ in such a way that $i<j$ holds whenever $i\in Q$ and $j\in Q'$. This yields
 \[
  A^\sqcup_{ij}=\begin{cases}
                 A_{ij}, & i,j\in Q\\
                 A'_{ij}, & i,j\in Q'\\
                 0, & \text{otherwise}
                \end{cases},
 \]
 which together with an order on $Q\sqcup Q'$ chosen as above gives the desired block-diagonal form, i.e. $\Phi(\cU)\oplus\Phi(\cU')=A\oplus A'\sim_p A^\sqcup=\Phi(\cU\sqcup\cU')$. Multiplicativity of $\Phi$ will be shown in the two lemmas below.
\end{proof}

\begin{lemma}\label{lem:Aprod1}
 For any $(i,j),(k,l)\in Q\times Q'$, using the notation from above, we have
 \begin{equation}\label{eq:Aprod}
  A^\times_{(k,l),(i,j)}=\mathds{1}_{\{j=l\}}A_{ki}+\one_{\{i=k\}}A'_{lj}.
 \end{equation}
\end{lemma}

\begin{proof}
 Let $\xi_{i}, \xi'_{j}$ and $\xi^\times_{(i,j)}$ be random variables distributed according to $\mu_i, \mu'_j$ and $\mu^\times_{(i,j)}$, respectively. Then, by conditioning on whether the urn evolves according to its first or second factor, we obtain
 \begin{align*}
  A^\times_{(k,l),(i,j)}
  &=a^\times_{(i,j)}\ev\left[\xi^\times_{(i,j),(k,l)}\right]\\
  &=(a_i+a'_j)\left(\one_{\{j=l\}}\frac{a_i}{a_i+a'_j}\ev[\xi_{ik}]
     +\one_{\{i=k\}}\frac{a'_j}{a_i+a'_j}\ev[\xi'_{jl}]\right)\\
  &=\mathds{1}_{\{j=l\}}A_{ki}+\one_{\{i=k\}}A'_{lj}. \qedhere
 \end{align*}
 In the case $a_i=a'_j=0$, both sides of \eqref{eq:Aprod} trivially vanish. 
\end{proof}

\begin{lemma}\label{lem:Aprod2}
 Let $q$ and $q'$ denote the number of types in $\cU$ and $\cU'$, respectively. Then, 
 \begin{equation}\label{eq:k+}
  A^\times \sim_p A\otimes I_{q'}+I_q\otimes A' = A\boxplus A'. 
 \end{equation}
\end{lemma}

\begin{proof}
 Note first that conjugating with a permutation matrix is the same as permuting the basis elements of the vector space or module on which the matrix acts as an endomorphism. Therefore, it suffices to show equality on the left side of \eqref{eq:k+} for one fixed ordering on the basis $Q\times Q'$ of $\Zsp(Q\times Q')$. To this end, and to simplify notation, assume w.l.o.g. $Q=\{1,2,...,q\}$ and $Q'=\{1,2,...,q'\}$ and equip both sets with the usual linear order $<$. This induces the lexicographic order on $Q\times Q'$ which is again linear. Assume that $A^\times$ is the matrix belonging to this choice of ordered basis.
 
 Proceed by splitting $A^\times$ into $q^2$ blocks of size $q'\times q'$. Due to the choice of ordering, the $i$-th block on the diagonal contains precisely the replacements of $(i,j)$ by $(i,l)$. By applying Lemma \ref{lem:Aprod1} to this block (and keeping in mind that on the diagonal, both summands of equation \eqref{eq:Aprod} contribute), we conclude that the $i$-th diagonal block equals $A_{ii}I_{q'}+A'$.
 
 Suppose now $i\neq k$ for $i,k\in Q$. Then, the $(i,k)$-th block of $A^\times$ has a non-zero entry at position $(i,j),(k,l)$ only if $j=l$, so on the diagonal of the block. This entry has to be -- again, due to Lemma \ref{lem:Aprod1} -- precisely $A_{ik}$. Thus, the $(i,k)$-th block equals $A_{ik}I_{q'}$.
 
 Putting everything together, we obtain
 \begin{align*}
  A_{\cU\times\cU'}&=
  \begin{pmatrix}
   A_{11}I_{q'}+A' & A_{12}I_{q'} & \cdots & A_{1q}I_{q'}\\
   A_{21}I_{q'} & A_{22}I_{q'}+A' & \cdots & A_{2q}I_{q'}\\
   \vdots & \vdots & \ddots & \vdots \\
   A_{q1}I_{q'} & A_{q2}I_{q'} & \cdots & A_{qq}I_{q'}+A'
  \end{pmatrix}\\
  &=A\otimes I_{q'} + I_q\otimes A',
 \end{align*}
 which concludes the proof of the lemma.
\end{proof}

By invoking Theorem 4.4.5 from \cite{HJ91}, we get as immediate consequence the following characterisation of the spectrum of $A^\times$:

\begin{cor}\label{cor:ghost}
 Respecting both algebraic and geometric multiplicity of eigenvalues, we have $\gs(A)+\gs(A')=\gs(A^\times)$, where the left-hand side has to be interpreted as Minkowski-addition, i.e. as forming the sum over all possible pairs of elements, including multiplicities. 
\end{cor}

The spectrum of a matrix is a (finite) multiset, which in turn is nothing but a (finitely supported) map $m:\IC\to\IN_{0}$, and can be represented as a finite formal sum $S_m(x):=\sum_{z\in\IC} m(z)x^z$. Then, taking the disjoint union of multisets $m,m'$ corresponds to summing those representations, i.e. $S_{m\sqcup m'}(x)=S_m(x)+S_{m'}(x)$, whereas the Minkowski addition between $m$ and $m'$ of Corollary~\ref{cor:ghost} corresponds to the product of two formal sums, i.e. $S_{m+ m'}(x)=S_m(x)S_{m'}(x)$. In particular, this implies that under these operations, the set of finite complex multisets, denoted $\fmSet_\IC$, is yet another semiring where the empty set and the multiset represented by $x^0$, respectively, are the neutral elements. 

Corollary \ref{cor:ghost} then states that the map $\gs$ assigning to $A\in\IntMatp$ (even in $\Matp$, see appendix A) its spectrum is multiplicative. It can easily be checked that it is also additive and preserves both the zero and the unit elements, and therefore is another morphism of semirings. Thus, we obtain the pipeline promised in Theorem~\ref{thm:pipeline}:
\begin{center}
 \begin{tikzcd}
  \PU \arrow[r, "\Phi"] & \IntMatp \arrow[r, "\sigma"] & \fmSet_\IC
 \end{tikzcd}
\end{center}

\section{Some additional properties of product urns} \label{sec:properties}

The purpose of this section is to highlight some consequences of Theorem~\ref{thm:pipeline} for properties that are often desirable for \Polya urns. We follow the notation of \cite{Ja04}.

\begin{defin}\label{defin:dom}
 If $Q$ is the set of colours of $\cU$, we say that $i$ \emph{dominates} $j$, $i\succ j$, for $i,j\in Q$ if $(A_{ji})^n>0$ for some $n\geq0$. 
\end{defin}
The relation $\succ$ is reflexive and transitive, it therefore contains an equivalence relation such that two colours $i,j$ belong to the same equivalence class if and only if both $i\succ j$ and $j\succ i$. Additionally, $\succ$ induces a partial order on the set of equivalence classes, we say that the maximal class is dominating, if it exists. By choosing a linear extension of this partial order, we can permute the colours in such a way that $A$ is block triangular, hence $\gs(A)$ is the union of the spectra of the diagonal submatrices (which are the restrictions of $A$ to the colours contained in the same equivalence class). Therefore, we call an eigenvalue dominating if it is an eigenvalue of the restriction of $A$ to the dominating class. Finally, we say $\cU$ is \emph{irreducible} if every colour is dominating, that is, if there is only one equivalence class.
 
We will mainly be concerned with the following properties, where the numbering of the assumptions is such that it is consistent with Janson's paper \cite{Ja04}:
\begin{description}
 \item[(A1)] The replacements satisfy $\xi_{i,j}(0)+\gd_{ij}\geq0$ almost surely for all colours $i,j\in Q$.
 \item[(A2)] The replacements satisfy $\ev\left[\xi_{i,j}(0)^2\right]<\infty$ for all $i,j\in Q$.
 \item[(A3)] The largest real eigenvalue $\gl_1\in\gs(A)$ is positive. 
 \item[(A4)] The largest real eigenvalue is simple.
 \item[(A5)] There exists a dominating colour $i$ with $X_i(0)>0$.
 \item[(A6)] The largest real eigenvalue is dominating.
\end{description}
Note that assumption (A1) is already encoded in our definition of \Polya urns via \eqref{eq:supp} and therefore redundant here. Moreover, since the only negative entries in $A$ lie on the main diagonal, there is some scalar $\ga$ such that $A+\ga I_q$ is a non-negative matrix. Then Perron-Frobenius theory guarantees the existence of a real eigenvalue $\gl_1$ such that all further eigenvalues can be ordered descendingly by their real parts: $\gl_1\geq\Re \gl_2\geq\Re \gl_3\geq\dots$. 

For a more detailed discussion of the importance and consequences of these assumptions, see \cite{Ja04} and references therein, such as \cite{Se06} and \cite{kesten1967limit}.

Just as in \eqref{eq:2urns}, we will extend our notation to urns $\cU$ and $\cU'$, with data decorated with a $'$ belonging to $\cU'$ (so e.g. $\gl_1$ and $\gl'_1$ being the largest real eigenvalues of $\cU$ and $\cU'$, respectively). 

\begin{prop}\label{prop:assumptions}
 For two \Polya urns $\cU, \cU'$, the following statements hold:
 \begin{enumerate}[(a)]
  \item If $\cU$ has equivalence classes $\cC_1,\cC_2,\dots, \cC_\nu$ (w.r.t. the relation just introduced) and $\cU'$ has equivalence classes $\cC'_1,\cC'_2,\dots,\cC_{\nu'}$, then $\cU\times\cU'$ has equivalence classes of the form $\cC_s\times \cC'_t$ for $1\leq s\leq\nu$ and $1\leq t\leq\nu'$. In particular, $\cU\times\cU'$ is irreducible if and only if $\cU$ and $\cU'$ are. Additionally, a class $\cC_s\times \cC'_t$ is dominating if and only both $\cC_s$ and $\cC_t$ are dominating in their respective urns. 
  \item If $\cU$ and $\cU'$ both satisfy the assumption (AX) for some $X\in\{1,\dots,6\}$, then so does $\cU\times\cU'$.
 \end{enumerate}
\end{prop}
\begin{proof}
 We first observe that
 \[
  (A\otimes I_{q'})(I_q\otimes A') = A\otimes A' = (I_q\otimes A')(A\otimes I_{q'})
 \]
 by the mixed product property for the Kronecker product. Therefore $A\otimes I_{q'}$ and $I_q\otimes A'$ commute with respect to the usual matrix multiplication, and we can apply the binomial formula to yield
 \begin{equation}\label{eq:power}
  (A\krs A')^n 
  = \sum_{k=0}^n \binom{n}{k} (A\otimes I_{q'})^{n-k}(I_q\otimes A')^k
  = \sum_{k=0}^n \binom{n}{k} A^{n-k}\otimes (A')^k.
 \end{equation}
 Now assume that $i\succ k$ in $\cU$ and $j\succ l$ in $\cU'$, and let $n$ and $n'$ respectively be the smallest non-negative integers such that $A^n_{ji}>0$ and $A^{n'}_{lk}>0$. Then, $(i,j)\succ(k,l)$ since
 \begin{align}\label{eq:xdom}
  (A^\times)^{n+n'}_{(i,j),(k,l)} 
  &= \sum_{m=0}^{n+n'} \binom{n+n'}{m} A^{n+n'-k}_{ki}\cdot(A')^k_{lj}\notag\\
  &= \binom{n+n'}{n'} A^n_{ki} \cdot (A')^{n'}_{lj} >0
 \end{align}
 by virtue of equation~\eqref{eq:power}. Conversely, if $i$ is not dominating over
 $k$ or $j$ is not dominating over $l$ (w.l.o.g. we assume the former), then $A^n_{ki}=0$ for all $n\geq0$, and the analogous computation to \eqref{eq:xdom} shows that $(A^\times)^n_{(k,l),(i,j)}$ vanishes for all $n$ as well. Thus, $(i,j)$ dominates $(k,l)$ if and only if both $i$ and $j$ dominate $k$ and $l$, respectively, and the claims in (a) follow. 
 
 \textbf{(A1)} This was already remarked when constructing the measures $\mu^\times_{(i,j)}$, immediately above Definition~\ref{defin:x}.
 
 \textbf{(A2)} Consider colours $(i,j)$ and $(k,l)$ from $\cU\times\cU'$. Similarly to the proof of Lemma~\ref{lem:Aprod1}, by conditioning on the urn according to which the ball $(i,j)$ evolves upon being drawn, we obtain:
 \begin{equation*}
  \xi^\times_{(i,j),(k,l)}(0)
  \stackrel{\scr D}{=} \one_{\{j=l\}}\frac{a_i}{a_i+a'_j}\xi_{i,k}(0)+\one_{\{i=k\}}\frac{a'_j}{a_i+a'_j}\xi'_{j,l}(0).
 \end{equation*}
 The claim now follows because a linear combination of $L^2$-random variables is again $L^2$.
 
 \textbf{(A3)} By Corollary~\ref{cor:ghost} the spectrum of $A^\times$ consists of all sums of the form $\gl+\gl'$, where $\gl$ and $\gl'$ are eigenvalues of $A$ and $A'$, respectively. Hence the eigenvalue $\gl_1^\times$ for $A^\times$ with the largest real part is given by $\gl_1+\gl'_1$ which by assumption (A3) is the sum of two positive reals, and therefore itself positive. 
 
 \textbf{(A4)} This follows similarly to the verification of (A3); in fact, the real part of the second eigenvalue of $A^\times$ is given by 
 \[
  \Re\gl^\times_2=\max\{\Re(\gl_1+\gl'_2),\Re(\gl_2+\gl'_1)\}.
 \]
 \textbf{(A5)} Let $i$ be a dominating colour of $\cU$ with $X_i(0)>0$, and let $j$ be a dominating colour of $\cU'$ with $X'_j(0)>0$. According to part (a), $(i,j)$ is then a dominating colour of $\cU\times\cU'$, and  we have $X^\times_{(i,j)}(0)=X_i(0)X'_j(0)>0$ by Definition~\ref{defin:x}. 
 
 \textbf{(A6)} Denote by $\cC_1$ and $\cC'_1$ the dominating classes of $\cU$ and $\cU'$ respectively. Assumption (A6) then implies that there is an eigenvector $v$ to $\gl_1$ and an eigenvector $v'$ to $\gl'_1$ such that $v_i=0$ for all $i\notin\cC_1$ and $v'_j=0$ for all $j\notin\cC'_1$. By Theorem 4.4.5 in \cite{HJ91} we know that $v\otimes v'$ is an eigenvector to the eigenvalue $\gl^\times_1=\gl_1+\gl'_1$ of $A^\times=A\krs A'$. Since $(v\otimes v')_{(i,j)}=v_i v'_j=0$ for $(i,j)\notin \cC_1\times\cC'_1$, we conclude that $\gl^\times_1$ is an eigenvalue to the dominating class, as required.
\end{proof}

\begin{rem}
 It is well-known (see e.g. \cite{athreya1969limit}) that the asymptotic behaviour of the urn process depends on the relation between the largest two eigenvalues. To be precise, one will get different results for the cases $\gl_1>2\Re \gl_2$, $\gl_1=2\Re \gl_2$, and $\gl_1<2\Re\gl_2$, see e.g. Corollaries 3.16-18 and Theorems 3.22-24 in \cite{Ja04}. As noted above, for the eigenvalues of the product urn, we have $\gl^\times_1=\gl_1+\gl'_2$, but $\gl^\times_2\in\{\gl_1+\gl'_2,\gl_2+\gl'_1\}$. Hence, even if both $\cU$ and $\cU'$ have small second eigenvalues, this does not automatically transfer to $\cU\times\cU'$.  
\end{rem}

However, we have:

\begin{cor}\label{cor:limit}
 Suppose $\cU$ and $\cU'$ satisfy (A1)-(A6). Let $v_1$ and $v'_1$ be the eigenvectors to $\gl_1$ and $\gl'_1$ that are normalized such that $\langle a,v_1\rangle = \langle a',v'_1\rangle=1$. Then, conditioned on essential non-extinction, 
 \[
  n^{-1}X^\times(n) \to S^{-1}(\gl_1+\gl'_1)(v_1\otimes v'_1)
 \] 
 almost surely for $n\to\infty$. Here, $S$ is the normalizing constant $S:=\langle \mathbf{1},v_1\rangle + \langle \mathbf{1},v'_1\rangle$, where $\mathbf{1}$ denotes the vector having entry 1 everywhere.
\end{cor}
\begin{proof}
 Since under the assumptions $\cU\times\cU'$ satisfies (A1)-(A6) by Proposition~\ref{prop:assumptions}, it follows from Section V.9.3 in \cite{athreya1972branching} (although we use the version stated as Theorem 3.21 in \cite{Ja04} for consistency) that, conditioned on essential non-extinction, $n^{-1}X^\times(n) \to \gl^\times_1 v^\times_1$ a.s. for $n\to\infty$, where $v^\times_1$ is normalized such that $\langle a^\times, v^\times_1\rangle=1$. Since $\gl^\times_1=\gl_1+\gl'_1$ and because the eigenspace to $\gl^\times_1$ is spanned by $v_1\otimes v'_1$, it only remains to check that $S$ is the correct scaling factor. And indeed,
 \begin{align*}
  \langle a^\times, v_1\otimes v'_1\rangle 
  &= \sum_{(i,j)\in Q\times Q'} (a_i+a'_j)v_{1,i}v'_{1,j}\\
  &= \sum_{j\in Q'} v'_{1,j} \sum_{i\in Q} a_i v_{1,i} + \sum_{i\in Q} v_{1,i} \sum_{j\in Q'} a'_j v'_{1,j} \\
  &= \langle \mathbf{1},v'_1\rangle + \langle \mathbf{1},v_1\rangle,
 \end{align*}
 using the condition that $v_1$ and $v'_1$ were appropriately normalized. 
\end{proof}

\begin{rem}
 Another obstacle in trying to transfer limit statements to product urns are the expressions for the asymptotic covariance matrix $\gS^\times$ (cf. \cite[(3.19) and (3.20)]{Ja04}), crucially involving the term
 \[
  \int_0^s e^{tA}\otimes e^{tA'}\Di t,
 \]
 and it seems unlikely that there will be a closed expression for $\gS^\times$ in terms of the corresponding matrices $\gS$ and $\gS'$. Nonetheless, we will mention two expressions that might be helpful for computing the covariance of $\cU\times\cU'$:
 
 Continuing the notation of Corollary~\ref{cor:limit}, define matrices $B_i:=\ev[\xi_i\xi_i^T]$ and $B:=\sum_{i\in Q} v_{1,i}a_iB_i$. Let $\cE^{(s,t)}_{i,j}$ be the $s\times t$-matrix whose $(i,j)$-th entry is 1, with all other entries being zero. Then
 \begin{equation}\label{eq:Bxij}
  B^\times_{(i,j)} = \frac{a_i}{a_i+a'_j} B_i\otimes \cE^{(q'\times q')}_{j,j}
   + \frac{a'_j}{a_i+a'_j} \cE^{(q\times q)}_{i,i}\otimes B'_j
 \end{equation}
 and 
 \begin{equation}\label{eq:Bx}
  B^\times = \frac{1}{S}\left(B\otimes\diag(v'_1) + \diag(v_1)\otimes B'\right).
 \end{equation}
 We omit the details, pointing out that equation~\eqref{eq:Bxij} follows by the same methods as Lemma~\ref{lem:Aprod1}, and implies \eqref{eq:Bx}.
\end{rem}

\section{Two special cases} \label{sec:examples}

In this section, we look at two special cases of product urns and present constructions that could potentially be useful in applications. 

\subsection{Slowing down \Polya urns}\label{ssec:slow} As stated in Theorem~\ref{thm:PUsr}, the urn $\cP_0$, consisting of one colour with activity 0, and starting with one ball, is the neutral element of multiplication of \Polya urns. 

We now consider the urn 
\begin{equation}\label{eq:Palpha}
 \cP_\ga := \left(\{*\}, \gd_0, \ga, 1\right)
\end{equation}
where $\ga\geq0$ is a parameter. For $\ga=0$, we obtain precisely the urn $\cP_0$ introduced above, and therefore have $\cU\times\cP_0 \cong \cP_0\times\cU \cong \cU$ for any \Polya urn $\cU$. For any other $\alpha$, there is a positive chance that a ball of colour $(i,*)$ drawn from $\cU\times\cP_\ga$ evolves according to the urn $\cP_\ga$, which means that the ball is simply returned back to the urn. This happens, according to \eqref{eq:xmeas}, exactly with probability $\frac{\ga}{a_i+\ga}$. In the special case where all activities in $\cU$ are the same positive number (w.l.o.g. 1), this probability does not depend on $i$, and on average every $(1+\ga)$-th draw will not change the current distribution of balls in the urn, thus effectively slowing the urn process down. 

Moreover, there is an embedding of urn processes in this case: Define the stopping times $\tau(k)$ to be the $k$-th time a ball drawn from $\cU\times\cP_\ga$ evolves according to $\cU$. Then, observing the urn process $X^\times$ for $\cU\times\cP_\ga$ at any time $n$ reveals that $X^\times(n)=X^\times(n-1)$ holds deterministically on the event $\bigcap_{k\geq0}\{\tau(k)\neq n\}$. In particular, the distribution of balls at time $\tau(k+1)$ is exactly the distribution of balls after $\tau(k)$. Therefore, the processes $(X(n))_{n\geq 0}$ and $\left(X^\times(\tau(n))\right)_{n\geq0}$ coincide in distribution.

\subsection{Random walks on product graphs} Let $G=(V,E)$ and $G'=(V',E')$ be finite simple graphs. Recall that the Cartesian product $G\square G'$ is the graph having the vertex set $V\times V'$, where vertices $(i,j)$ and $(k,l)$ are neighbours if and only if $i=k$ and $j$ and $l$ are neighbours in $G'$, or $j=l$ and $i$ and $k$ are neighbours in $G$. Also recall that a simple random walk on a (finite simple) graph $G$ is a Markov process $(X(n))_{n\geq0}$ on the state space $V$, where $X(n+1)$ is chosen uniformly at random among the neighbours of $X(n)$, independently of the previous history of the process. Any such random walk starting at $v_0\in V$ can be thought of as a generalised \Polya urn, where the colour set is given by $V$, the activity of a colour is set to be the degree of the corresponding vertex, the initial distribution is one ball of colour $v_0$, and the replacement rules are such that if a ball of colour $v\in V$ is drawn, it is replaced by a ball of colour $w$, where $w\in V$ is chosen at random among the neighbouring vertices of $v$. In particular, there will always be exactly one ball in the urn. Denote an urn constructed from the simple random walk on $G$ that started in $v_0$ by $\cU_{G,v_0}$. 

\begin{prop}\label{prop:xgraph}
 With the notation just introduced, the urns $\cU_{G,v_0}\times \cU_{G',v'_0}$ and $\cU_{G\square G',(v_0,v'_0)}$ are strictly isomorphic for any starting states $v_0\in V, v'_0\in V'$. In particular, the corresponding urn processes are equally distributed. 
\end{prop}
\begin{proof}
 By construction and Definition~\ref{defin:x}, both urns have the colour set $V\times V'$ and start with one ball of colour $(v_0,v'_0)$. Equality of the activity vectors follows from the fact that $\deg_{G\square G'}(v,v')=\deg_G(v) + \deg_{G'}(v')$. It remains to check that the replacements coincide. So, consider $v,w\in V$ and $v',w'\in W$. In both urns, a ball of colour $(v,v')$ cannot be replaced by a $(w,w')$-coloured ball if $(v,v')=(w,w')$ or if $v\neq w$ and $v'\neq w'$. Assume $v=w$ and $v'\neq w'$. The urn $\cU_{G\square G',(v_0,v'_0)}$ replaces $(v,v')$ by $(v,w')$ with probability $\frac{1}{\deg_G(v)+\deg_{G'}(v')}$. On the other hand, for $\cU_{G,v_0}\times \cU_{G',v'_0}$ to perform the same replacement, the ball with colour $(v,v')$ first needs to evolve according to $\cU_{G',v'_0}$ which happens with probability $\frac{\deg_{G'}(v')}{\deg_G(v)+\deg_{G'}(v')}$, and then (independently of it) $v'$ needs to be replaced by $w'$ in $\cU_{G',v'_0}$ which happens with probability $\frac{1}{\deg_{G'}(v'_0)}$. Therefore, the replacement probabilities agree with one another, and the case $v\neq w$ and $v'=w'$ can be treated analogously. 
\end{proof}


\appendix
\section*{Appendix. Two matrix semirings} 

We use this appendix to give a short overview of the algebraic structure of square matrices under permutation similarity. While this is most likely not a novel result, the author has not been able to find it in the literature either. 

\begin{prop}\label{prop:Matp}
 The set $\Matp:=\Mat/\sim_p$ of equivalence classes of all square matrices under permutation similarity is a commutative semiring with addition $\oplus$ and multiplication $\krs$. Its zero-element is given by the (equivalence class of the) empty matrix $()$ and the neutral multiplicative element is the (equivalence class of the) $1\times1$-matrix $(0)$.
\end{prop}
\begin{proof}
 For the entire proof, let $A,B,C$ be matrices of size $p\times p, q\times q$ and $r\times r$, respectively. Let $P$ and $Q$ always denote permutation matrices of the required size. 
 
 We show first that the operations are indeed well-defined. For this, let $A'=P^{-1}AP$ and $B'=Q^{-1}BQ$. Then, one easily checks that $A'\oplus B'=(P\oplus Q)^{-1}(A\oplus B)(P\oplus Q)$ and, by using the mixed product property, that
 \begin{align*}
  A'\krs B'
  &=(P^{-1}AP)\krs (Q^{-1}BQ)\\
  &=(P^{-1}AP)\otimes (Q^{-1}I_qQ)+ (P^{-1}I_pP)\otimes (Q^{-1}BQ)\\
  &=(P^{-1}\otimes Q^{-1})(A\otimes I_q)(P\otimes Q)
     +(P^{-1}\otimes Q^{-1})(I_p\otimes B)(P\otimes Q)\\
  &=(P\otimes Q)^{-1}(A\krs B)(P\otimes Q).
 \end{align*}
 
 It is straightforward to verify that $0=(\ )$ acts as it should, i.e. that $A\oplus (\ )=(\ )\oplus A=A$ and $A\krs(\ )=(\ )\krs A=(\ )$. Moreover, $A\krs (0)=A\otimes I_1+I_p\otimes (0)=A\otimes I_1=A$ and analogous for $(0)\krs A$. Clearly, $\oplus$ is associative. The commutativity of $\oplus$ relies on $A\oplus B=P^{-1}(B\oplus A)P$, where $P$ is the matrix of block form
 \[
  P=\begin{pmatrix} 0 & I_q \\ I_p & 0 \end{pmatrix}.
 \]
 
 Associativity of the multiplication $\krs$ follows from direct calculation. For commutativity, we use that the permutation matrix for $A\otimes B\sim_p B\otimes A$ depends exclusively on the size of $A$ and $B$, but not on their entries (cf. Theorem 9.13 and its proof in \cite{BC08}). That is, there exists a $pq\times pq$-sized permutation matrix $P$ such that 
 \[
  A\krs B=P^{-1}(I_q\otimes A)P+P^{-1}(B\otimes I_p)P=P^{-1}(B\krs A)P.
 \]
 For distributivity, one checks that $A\krs(B\oplus C)=P^{-1}((A\krs B)\oplus(A\krs C))P$ where $P$ is of the block form
 \[
  P=\begin{pmatrix}
     I_q & 0 & 0 & 0 &  \cdots & 0 & 0 \\
     0 & 0 & I_q & 0 &  \cdots & 0 & 0 \\
     \vdots & \vdots & & & \cdots & & \vdots \\
     0 & 0 & 0 & 0 &        & I_q & 0 \\
     0 & I_r & 0 & 0 & \cdots & 0 & 0 \\
     0 & 0 & 0 & I_r &        & 0 & 0 \\
     \vdots & \vdots & & & \cdots & & \\
     0 & 0 & 0 & 0 & \cdots & 0 & I_r
    \end{pmatrix}
 \]
 
 Hence, $(\Matp,\oplus,\krs)$ is a commutative semiring. 
\end{proof}

Conjugating a matrix by a permutation matrix does not exchange diagonal and off-diagonal entries. Hence the quotient $\IntMat /\sim_p$ is well-defined, and we denote this set by $\IntMat_p$. In this way, we obtain: 

\begin{cor}\label{cor:Intmp}
 $(\IntMatp,\oplus,\krs)$ is a sub-semiring of $(\Matp,\oplus,\krs)$.
\end{cor}

\subsection*{Acknowledgements} The author wishes to thank his academic advisors, Cecilia Holmgren and Svante Janson, for their generous support and many helpful remarks and discussions. This work was partially supported by grants from the Knut and Alice Wallenberg Foundation, the Ragnar Söderberg Foundation, and the Swedish Research Council.


\bibliographystyle{alphaurl}
\bibliography{produrnsbib}

\end{document}